\def\timestamp{%
Time-stamp: <on-Hstar.tex: Monday 13-01-2014 at 21:47:40 (cet)>}
\def\stripname Time-stamp: <#1 #2>{#2}
\edef\filedate{\expandafter\stripname\timestamp}

\documentclass[a4paper]{amsart}

\usepackage{amsrefs}

\DeclareMathSymbol\HH 0{AMSb}{`H}
\DeclareMathSymbol\II 0{AMSb}{`I}
\DeclareMathSymbol\M  0{AMSb}{`M}
\DeclareMathSymbol\Q  0{AMSb}{`Q}
\DeclareMathSymbol\R  0{AMSb}{`R}
\DeclareMathSymbol\Z  0{AMSb}{`Z}

\newcommand\betaH{\beta\HH}
\newcommand\Hstar{\HH^*}
\newcommand\Mstar{\M^*}
\newcommand\Nstar{\omega^*} \let\N\omega

\DeclareMathSymbol\restr\mathbin{AMSa}{"16}  
\DeclareMathSymbol\forces\mathrel{AMSa}{"0D}
\DeclareMathSymbol\le    \mathrel{AMSa}{"36}
\DeclareMathSymbol\ge    \mathrel{AMSa}{"3E}
\DeclareMathSymbol\preceq\mathrel{AMSa}{"34}

\newcommand\calA{\mathcal{A}}
\newcommand\calC{\mathcal{C}}
\newcommand\calF{\mathcal{F}}
\newcommand\Pow{\mathcal{P}}
\newcommand\calS{\mathcal{S}}

\newcommand\axiom{\mathsf}
\newcommand\CH{\axiom{CH}}
\newcommand\ZFC{\axiom{ZFC}}

\newcommand\orpr[2]{\langle{#1},{#2}\rangle}
\newcommand\bigorpr[2]{\bigl<{#1},{#2}\bigr>}
\newcommand\trip[3]{\langle{#1},{#2},{#3}\rangle}
\newcommand\btrip[3]{\bigl<{#1},{#2},{#3}\bigr>}
\newcommand\card[1]{\mathopen|{#1}\mathclose|}

\newcommand\dom{\operatorname{dom}}
\newcommand\cl{\operatorname{cl}}
\renewcommand\Im{\operatorname{Im}}

\newcommand\cee{\mathfrak{c}}

\newcommand{\Ex}{\operatorname{Ex}}
\newcommand{\Fn}{\operatorname{Fn}}
\newcommand\CE{C_E}

\newcommand\functions[2]{\vphantom{#2}^{#1}#2}

\let\fname\dot

\newcommand\preim{^\gets}
\newcommand\concat{\mathbin{{}^\frown}}
\newcommand\omegaseq[2][n]{\langle{#2}_{#1}:#1\in\omega\rangle}
\newcommand\omegaoneseq[1]{\langle{#1}_\alpha:\alpha\in\omega_1\rangle}

\newcommand\abseq[1][]{\bigl<[a_n^{#1},b_n^{#1}]:n\in\omega\bigr>}
\newcommand\A[2]{A({#1},{#2})}
\newcommand\bA[2]{A\bigl({#1},{#2}\bigr)}

\newcommand\good{}
\newcommand\Good{}
\newcommand\goodt{}
\newcommand\Goodt{}
\def\good/{bad}
\def\Good/{Bad}
\def\goodt/{\good/ triple}
\def\Goodt/{\Good/ triple}

\theoremstyle{plain}
\newtheorem{theorem}{Theorem}[section]
\newtheorem{lemma}[theorem]{Lemma}
\newtheorem{proposition}[theorem]{Proposition}

\theoremstyle{definition}

\theoremstyle{remark}
\newtheorem{remark}{Remark}[section]

\newcounter{case}
\newenvironment{case}[1]{\par\smallbreak\stepcounter{case}
                      \noindent Case~\arabic{case}: \textsl{#1}.
                      \ignorespaces}{\par}

\begin{document}

\title{On subcontinua and continuous images of $\beta\R\setminus\R$}

\author[Alan Dow]{Alan Dow\dag}
\address{
Department of Mathematics\\
UNC-Charlotte\\
9201 University City Blvd. \\
Charlotte, NC 28223-0001}
\email{adow@uncc.edu}
\urladdr{http://www.math.uncc.edu/\~{}adow}
\thanks{\dag Research of the first author was supported by NSF grant 
             No.\ NSF-DMS-0901168.} 

\author{Klaas Pieter Hart}
\address{Faculty of Electrical Engineering, Mathematics and Computer Science\\  
         TU Delft\\
         Postbus 5031\\
         2600~GA {} Delft\\
         the Netherlands}
\email{k.p.hart@tudelft.nl}
\urladdr{http://fa.its.tudelft.nl/\~{}hart}


\date{\filedate}

\begin{abstract}
We prove that the \v{C}ech-Stone remainder of the real line has
a family of $2^\cee$ mutually non-homeomorphic subcontinua.

We also exhibit a consistent example of a first-countable continuum
that is not a continuous image of~$\Hstar$.
\end{abstract}

\maketitle

\section*{Introduction}

This paper contains two disparate results on $\Hstar$, 
the \v{C}ech-Stone remainder of the half~line $\HH=[0,\infty)$.

\smallskip

We prove that $\Hstar$ has a family of $2^\cee$~many
mutually non-homeomorphic subcontinua.
This completes the proof of this fact begun in~\cite{MR2823685};
in that paper the first-named author showed that that $\neg\CH$, 
the negation of the Continuum Hypothesis, 
implies that such a family exists, consisting of \emph{decomposable} continua.

We prove that $\CH$ also implies the existence of a family of $2^\cee$~many 
mutually nonhomeomorphic subcontinua as well; in fact, we construct,
in one fell swoop, two families: one consisting of indecomposable, the other
of decomposable continua. 

This suggests the obvious question whether one construct from~$\ZFC$, 
or even $\ZFC+\neg\CH$, a family of $2^\cee$~many mutually non-homeomorphic
indecomposable subcontinua of~$\Hstar$.

\smallskip

Our second result concerns continuous images of~$\Hstar$.
There are various parallels between~$\Hstar$ and~$\Nstar$ as regards
their continuous images.
Some of these can be found in~\cite{MR1707489}: every continuum of 
weight~$\aleph_1$ or less is a continuous image of~$\Hstar$ and
the Continuum Hypothesis implies that the continuous images of~$\Hstar$
are exactly the continua of weight~$\cee$ or less 
(parallel to Parovi\v{c}enko's results from~\cite{Parovicenko63}
 on continuous images of~$\omega^*$).
That not all results carry over was shown in~\cite{MR2425747}:
there is a continuum that is a continuous image of~$\omega^*$ (it is even
separable) that is consistently not a continuous image of~$\Hstar$.
Also, the Open Colouring Axiom implies that $\Hstar$~itself is not a 
continuous image of~$\omega^*$, see~\cite{MR1679586}.

We present another parallel, this one of Bell's result from~\cite{MR1058795}
that, consistently,
not every first-countable compact space is a continuous image of~$\omega^*$.
We give a consistent example of a first-countable continuum that is
neither a continuous image of~$\omega^*$ nor one of~$\Hstar$.
The interest in such examples stems from Arhangel$'$ski\u\i's theorem
in~\cite{MR0251695} that compact first-countable spaces have cardinality
and hence weight at most~$\cee$ and thus are continuous images of~$\omega^*$
if one assumes~$\CH$.

\section{Preliminaries}

In this section we collect the necessary results on the subcontinua
of~$\Hstar$ that we shall need.
We refer to~\cite{Hart} for the necessary proofs and further information.

\subsection{An auxiliary space}

A useful space to have is the product $\omega\times\II$, which we denote
by~$\M$.
Its \v{C}ech-Stone compactification, $\beta\M$, and its remainder, $\Mstar$,
are very useful in the study of $\beta\HH$ and~$\Hstar$ because there are many
continuous maps from both onto their respective counterparts.

The natural projection $\pi:\M\to\omega$ extends to a surjection
$\beta\pi:\beta\M\to\beta\omega$; because $\pi$~is monotone the
extension~$\beta\pi$ is monotone as well.
For $u\in\beta\omega$ we denote the preimage $\beta\pi\preim(u)$ by~$\II_u$.
For $n\in\omega$ we simply have $\II_n=\{n\}\times\II$ but if $u\in\omega^*$
then $\II_u$~is a continuum that has a few properties that make it
resemble~$\II$ somewhat.

It has two end points, $0_u$ and~$1_u$; these are obtained by 
intersecting~$\II_u$ with the closures of $\omega\times\{0\}$ 
and $\omega\times\{1\}$ respectively.
The continuum~$\II_u$ is irreducible between these end points and thus
it is divided into layers by the following quasi-order:
$x\preceq y$ iff every subcontinuum of~$\II_u$ that contains~$0_u$ and~$y$
also contains~$x$.
These layers are the equivalence classes under the equivalence relation
`$x\preceq y$ and $y\preceq x$' and they form an upper semicontinuous
decomposition of~$\II_u$ with an ordered continuum as its decomposition space.

Many of these layers are one-point sets, for instance: every sequence
$\omegaseq{x}$ in~$\II$ determines a point~$x_u$:
the unique point of~$\II_u$ that is in the closure of
the set $\bigl\{\orpr{n}{x_n}:n\in\omega\bigr\}$.
Each such point is a cut~point and the set of these is dense in~$\II_u$,
and linearly ordered by~$\preceq$.
If $\omegaseq{x}$ is an increasing sequence in~$\II_u$ then its `supremum'
is a single layer that is non-trivial since it contains the accumulation
points of~$\omegaseq{x}$ and these form a set that is homeomorphic
to~$\omega^*$, because $\Hstar$~is an $F$-space.
Also, every layer is an indecomposable continuum; this fact will make 
some verifications in our construction relatively painless.

\subsection{Subcontinua of $\Hstar$}
\label{subsec:subcont}

We now describe a general construction of subcontinua of~$\Hstar$.
To this end let $\abseq$ be a sequence of closed intervals in~$\HH$
such that $b_{n+1}=a_n$ for all~$n$ and $\lim_{n\to\infty}a_n=\infty$.
Take the map $q:\M\to\HH$ defined by $q(n,t)=a_n+t(b_n-a_n)$ for
all~$n$ and~$t$.
This map is almost everywhere one-to-one; the exceptions are at the 
end points: we always have $q(n,1)=q(n+1,0)$.
This behaviour persists when we take $\beta q$; this map is also almost
injective, the exceptions are that $\beta q(1_u)=\beta q(0_{u+1})$ for all~$u$,
where $u+1$ is the image of~$u$ under the extension of the shift 
map $n\mapsto n+1$.

For every $u\in\omega^*$ the restriction of $\beta q$ to~$\II_u$ is
injective and hence an embedding.
We shall denote the image by $[a_u,b_u]$ and refer to such a continuum
as a standard subcontinuum of~$\Hstar$.

These continua determine the structure of the other continua completely:
every subcontinuum of~$\Hstar$ is both the intersection and the union of
families of standard subcontinua.

Some work is needed to establish the following fundamental facts:

\begin{lemma}[\cite{Hart}*{Theorem~5.8}]\label{lemma.interval}
Every decomposable subcontinuum of $\Hstar$ is a non-trivial interval
in some standard subcontinuum.  \qed
\end{lemma}

\begin{lemma}[\cite{Hart}*{Theorem~5.9}]\label{lemma.comparable}
If $K$ and $L$ are subcontinua of~$\Hstar$ that intersect and if one
of these is indecomposable then $K\subseteq L$ or $L\subseteq K$. \qed  
\end{lemma}

In particular: if a standard subcontinuum $K$ intersects an 
indecomposable subcontinuum~$L$ then 
either $K\subseteq L$ and $K$~is nowhere dense in~$L$, 
or $L$~is contained in a layer of~$K$ and hence nowhere dense in~$K$.

\begin{lemma}[\cite{Hart}*{Theorem~5.10}]\label{lemma.K-int-L}
If $K$ and $L$ are subcontinua of~$\Hstar$ such that $K$~is a proper
subset of~$L$ and $L$~is indecomposable then there is a standard
subcontinuum~$M$ such that $K\subseteq M\subseteq L$. \qed  
\end{lemma}

\section{Getting the continua}
\label{sec:getting}

In this section we describe a general construction of indecomposable
continua in~$\Hstar$; in the next section we show that we can actually find
$2^\cee$~many such continua.

We let $\Gamma$ denote the collection of all sequences $\abseq$
of closed intervals in~$\HH$ with integer end points and such that
$b_n=a_{n+1}$ for all~$n$.

As we have seen above, if $A=\abseq$ is such sequence then for every free
ultrafilter on~$\omega$ we obtain the standard subcontinuum~$[a_u,b_u]$. 

We can also associate an other subcontinuum to~$A$ and an ultrafilter~$u$,
as follows.
If $q$~is the map from $\M$ to~$\HH$ associated to~$A$ as above then the 
restriction $\beta q\restr\Mstar$ maps~$\Mstar$ onto~$\Hstar$.
Therefore there is an ultrafilter~$v$ on~$\omega$ such that $u\in[a_v,b_v]$;
this continuum we shall denote by~$A_u$.

Thus each ultrafilter $u$ determines a whole family of continua in~$\Hstar$,
to wit $\calS_u=\{A_u:A\in\Gamma\}$.

We shall find $2^\cee$ many ultrafilters on $\omega$ and for each
such ultrafilter~$u$ a chain~$\calC_u$ in $\calS_u$.
Each chain~$\calC_u$ gives us an indecomposable continuum, 
$K_u=\cl\bigcup\calC_u$, and our ulterior motive is to have all~$K_u$ be
mutually non-homeomorphic.

To this end we shall find for each linear order $\orpr T\prec$ of 
cardinality~$\aleph_1$ an ultrafilter~$u_T$, in fact a $P$-point,
such that $T$ embeds in~$\calS_{u_T}$ in a special way:
there will be a family~$\{A^t:t\in T\}$ in~$\Gamma$ such that
\begin{enumerate}
\item $t\prec s$ iff $A^t_{u_T}$ is contained in a layer of~$A^s_{u_T}$ 
\item every $A\in\Gamma$ is equivalent to some $A^t$, in a manner
      to be specified presently\label{embed.precise}
\end{enumerate}
These two conditions will ensure that a homeomorphism between~$K_{u_T}$
and~$K_{u_S}$ will give rise to an isomorphism between final segments
of~$T$ and~$S$.
Thus the proof will be finished once we exhibit $2^\cee$ many linearly ordered
sets without isomorphic final segments.

As mentioned before, the construction proceeds under the assumption
of the Continuum Hypothesis.

\subsection{\Goodt/s}

The central notion will be that of a \goodt/.%
\footnote{The word `good' seems overused and, especially in the vernacular,
          `bad' may carry a positive connotation}

A \goodt/ has three coordinates:
\begin{itemize}
\item a free filter base $\calF$ on $\omega$,
\item a linear order $\orpr T\prec$, and
\item a subset $\calA_T=\{A^t:t\in T\}$ of $\Gamma$.
\end{itemize}
These should satisfy the following properties, where, in the interest of
readability we write $\A tn$ for $[a^t_n,b^t_n]$.
\begin{enumerate}
\item if $s\prec t$ in~$T$ then there is $F\in\calF$ such that for every~$k$
      there is an~$l$ for which $\A sk\cap F\subseteq \A tl$
      \label{bad.tr.1}
\item for every decreasing sequence $\langle t_i:i<l\rangle$ in $T$,
      for every $m\in\omega$ and every $F\in\calF$ there is a 
      function $\varphi:\functions{\le l}m\to\omega$ such that
      \label{bad.tr.2}
      \begin{enumerate}
      \item if $\rho\in\functions lm$ then $\varphi(\rho)\in F$,
      \item if $\rho\in\functions{<l}m$ then $i\mapsto \varphi(\rho\concat i)$
            is increasing \label{bad.tr.2.a}
      \item if $k<l$ and $\rho\in\functions km$
            then $\bA{t_{k+1}}{\varphi(\rho\concat i)}\subseteq
                  \bA{t_k}{\varphi(\rho)}$
            for all $i<m$. 
      \end{enumerate}
\end{enumerate}
If $\calF$ is an ultrafilter then property~\eqref{bad.tr.1} translates into
$A^s_\calF\subseteq A^t_\calF$ and property~\eqref{bad.tr.2} implies that the 
inclusion is as described above: the (possibly partial) function 
$\psi$ that satisfies $\psi(k)=l$ iff $\A sk \subseteq\A tl$
is finite to one, but its fibers have unbounded cardinality, even
when restricted to an arbitrary element of~$\calF$ and this implies that
$A^s_\calF$ is a subset of a layer of~$A^t_\calF$.

Condition~\eqref{bad.tr.2} will also be seen to keep our recursive 
constructions alive.
To be able to keep our formulations readable we shall say that the
function~$\varphi$ in this condition
is $m$-dense for~$F$ and~$\langle t_i:i<l\rangle$,
or for~$F$ and~$\{t_i:i<l\}$ (set rather than sequence).
We shall abbreviate $\{\varphi(\rho):\rho\in\functions lm\}$ as~$\Im\varphi$
and refer to it as the image of~$\varphi$.

The following is a sketch of the construction.
Let $\orpr T\prec$ be a linear order of cardinality~$\aleph_1$ and let
$\omegaoneseq t$ be an enumeration of~$T$.
By transfinite recursion we construct a sequence $\omegaoneseq F$
of infinite subsets of~$\omega$ and a map $t\mapsto A^t$ from~$T$ to~$\Gamma$
such that
\begin{enumerate}
\item $F_\beta\subseteq^* F_\alpha$ whenever $\alpha<\beta$
\item $\trip{\calF_\alpha}{T_\alpha}{\calA_\alpha}$ is a \goodt/,
      where $\calF_\alpha=\{F_\beta:\beta<\alpha\}$,
      $T_\alpha=\{t_\beta:\beta<\alpha\}$, and
      $\calA_\alpha=\{A^{t_\beta}:\beta<\alpha\}$
\item $\{F_\alpha:\alpha\in\omega_1\}$ generates an ultrafilter on~$\omega$.
\end{enumerate}
For technical reasons we add a minimum and a maximum to~$T$, if not already 
present.

We will formulate and prove a series of lemmas about \goodt/s that
will facilitate such a construction; the standing assumptions 
in the lemmas will be 
\begin{enumerate}
\item $\calF$ and $T$ are countable, and $\calF$ extends the cofinite filter,
\item $T$~has a minimum and a maximum, denoted $0$ and $1$ respectively, and
\item $\abseq[0]=\bigl<[n,n+1]:n\in\omega\bigr>$.
\end{enumerate}

To begin we show that at any time during our construction we can assume
that $\calF$~is a principal filter, or rather, the restriction of the cofinite
filter to a single set.

\begin{lemma}\label{lemma:diagonal}
If $\trip{\calF}{T}{\calA_T}$ is a \goodt/ then there is a single infinite $G$
such that $G\subseteq^*F$ for all $F\in\calF$ and such that
$\btrip{\{G\}}{T}{\calA_T}$ is a \goodt/.
\end{lemma}

\begin{proof}
Let $\omegaseq{T}$ be an increasing sequence of finite sets whose union
is~$T$ and let $\omegaseq F$ be a sequence in~$\calF$ such that for every
$F\in\calF$ there is an~$n$ such that $F_n\subseteq F$.
Recursively let $\varphi_m$ be $m$-dense for~$F_m$ and~$T_m$ 
and such that $\Im\varphi_m$~is disjoint from~$\Im\varphi_i$ for $i<m$.
Then $G=\bigcup_{m\in\omega}\Im\varphi_m$ is as required.
\end{proof}

This lemma is used at limit steps of our construction, basically to make
them look like successor steps.
We shall write $\trip{G}{T}{\calA_T}$ for $\btrip{\{G\}}{T}{\calA_T}$.

At some steps in the construction the following technical fact will be useful.

\begin{lemma}\label{lemma.disjoint.dense}
A triple $\trip{F}{T}{\calA_T}$ is \good/ if and only if for every (some)
increasing sequence~$\omegaseq m$ in~$\omega$ and every (some) increasing
sequence~$\omegaseq T$ finite subsets of~$T$ such that 
$T=\bigcup_{n\in\omega}T_n$
there is a sequence $\omegaseq\varphi$ of functions such that
$\varphi_n$~is $m_n$-dense for $F$ and~$T_n$, 
and $\max\Im\varphi_n<\min\Im\varphi_{n+1}$ for all~$n$.
\end{lemma}

\begin{proof}
For the non-trivial implication we find the functions $\varphi_n$
by recursion: $\varphi_0$~exists by assumption
and if $\varphi_n$ is found then we let $M=\max\Im\varphi_n$ and we choose
a function~$\varphi$ that is $M+m_{n+1}+1$-dense for $F$ and~$T_{n+1}$.
By condition~\eqref{bad.tr.2.a} in the definition of a \goodt/ we have
$\varphi(M+1+\rho)>M$ 
whenever $\rho\in\functions i{m_{n+1}}$ for some~$i\le|T_{n+1}|$
(here $M+1+\rho$ denotes the sequence obtained by adding $M+1$ to
all values of~$\rho$).
Thus defining $\varphi_{n+1}(\rho)=\varphi(M+1+\rho)$ gives us our next
function.
\end{proof}

The next lemma ensures that we can make our final filter an ultrafilter.

\begin{lemma}\label{lemma.split}
Let $\trip{F}{T}{\calA_T}$ be a \goodt/ and assume $F=F_0\cup F_1$;
then at least one of $\trip{F_0}{T}{\calA_T}$ and $\trip{F_1}{T}{\calA_T}$
is a \goodt/.
\end{lemma}

\begin{proof}
We show by induction on $l$: if $\langle t_i:i<l\rangle$ is decreasing
and $\varphi$~is $2m$-dense for~$F$ and~$\langle t_i:i<l\rangle$ 
then $\varphi$ induces an $m$-dense function for~$F_0$ or~$F_1$
and~$\langle t_i:i<l\rangle$.

If $l=1$ then $\Im\varphi$ is just a $2m$-element subset of~$F$ and its 
intersection with one of~$F_0$ and~$F_1$ has at least $m$~elements;
the increasing enumeration of that intersection is $m$-dense.

In the step from $l$ to $l+1$ we let $\langle t_i:i\le l\rangle$ 
and a $2m$-dense $\varphi$ be given.
For each $j<2m$ the function $\varphi_j:\functions{\le l}{2m}\to\omega$,
defined by $\varphi_j(\rho)=\varphi(j\concat\rho)$, is $2m$-dense
for~$F$ and~$\langle t_i:1\le i\le l\rangle$ and so induces an $m$-dense
function~$\varphi_j'$ for $F_{\epsilon_j}$ and~$\langle t_i:1\le i\le l\rangle$, 
where $\epsilon_j\in\{0,1\}$.
Take $\epsilon$ such that $A=\{j:\epsilon_j=\epsilon\}$ has size at least~$m$
and define $\varphi':\functions{\le l+1}m\to\omega$ by 
`$\varphi'(\langle j\rangle)$~is the $j$th element of~$A$'
and $\varphi'(j\concat\rho)=\varphi'_{\varphi'(\langle j\rangle)}(\rho)$ 
for $\rho\in\functions{\le l}m$.

Now enumerate $T$ as $\omegaseq t$ and apply the above for each $m$ to the pair 
$\langle t_i:i<m\rangle$ and~$m$.
Whichever of $F_0$ and $F_1$ appears infinitely often in the conclusion
is the set that we seek.
\end{proof}

Now we show how to extend the ordered set $T$ by one element.

\begin{lemma}
Let $\trip{F}{T}{\calA_T}$ be a \goodt/ and let $t^*$ be a point not in~$T$.
Assume $T\cup\{t^*\}$~is ordered so that $T$~retains its original order and
$0\prec t^*\prec1$.
Then there are $G\subseteq F$ and $A^{t^*}\in\Gamma$ such that
$\btrip{G}{\calA_T\cup\{A^{t^*}\}}{T\cup\{t^*\}}$ is a \goodt/.  
\end{lemma}

\begin{proof}
We write $T$ as an increasing union of finite sets $T_m$, with $0,1\in T_0$
and we construct $G$ and $A^{t^*}$ as follows.
We apply Lemma~\ref{lemma.disjoint.dense} to find a sequence
$\omegaseq[m]\varphi$ such that $\varphi_m$ is $m^2$-dense for $F$ and~$T_m$,
and $\max\Im\varphi_m<\min\Im\varphi_{m+1}$ for all~$m$.

We fix $m$ for the moment and let $\langle t_i:i<l\rangle$ enumerate~$T_m$
in decreasing order and let $i$ be such that $t_{i+1}\prec t^*\prec t_i$.
Our task is to convert $\varphi_m$ into an $m$-dense function for 
our future~$G$ and $T_m\cup\{t^*\}$.
The idea is simple
--- we use level~$i+1$ in~$\dom\varphi_m$ to create
 two levels in $\dom\psi_m$ --- but the notation is a bit messy: 
we take the following subset of the domain of~$\varphi_m$:
$$
D=\{\rho\in\dom\varphi_m:
    (\forall j\in\dom\rho)(j\neq i \Rightarrow \rho(j)<m)\}
$$
Using the $m^2$ values for all~$\rho(i)$ we transform $D$ into the tree 
$\functions{\le l+1}m$:
\begin{itemize} 
\item if $\dom\rho\le i$ then $\rho$ does not change;
\item if $\dom\rho=i+1$ then $\rho=\rho'\concat(km+ j)$ for some
  $\rho'\in\functions im$ and $k,j<m$; in this case $\rho$ determines two
  nodes: $\rho^+=\rho'\concat k$ and $\rho^{++}=\rho'\concat k\concat j$
\item if $i+1<\dom\rho$ then $\rho=\rho'\concat(km+ j)\concat\sigma$ for
  some $\rho'\in\functions im$, some $k,j<m$ and some sequence $\sigma$;
  then $\rho$ determines $\rho^+=\rho'\concat k\concat j\concat\sigma$.
\end{itemize}
We define $\psi_m:\functions{\le l+1}m\to\omega$ by 
$$
\psi_m(\varrho)=
\begin{cases}
\varphi_m(\varrho) & \text{ if }\dom\varrho\le i\\
\varphi_m(\rho) & \text{ if }\varrho=\rho^{++} \text { for some }
                  \rho\in\functions{i+1}m\\
\varphi_m(\rho) & \text{ if }\varrho=\rho^+ \text { for some }
                  \rho \text{ with }\dom\rho>i+1\\
\end{cases}
$$
This leaves $\psi_m(\varrho)$ undefined in case $\dom\varrho=i+1$,
that is, if $\varrho=\rho\concat k$ for some $\rho\in\functions im$
and $k<m$,
and it is here that we build and insert part of $A^{t^*}$.

In words: for each $\rho\in\functions im$ we bundle the $m^2$ intervals 
$[a^{t_{i+1}}_{\varphi_m(\rho\concat j)},
  b^{t_{i+1}}_{\varphi_m(\rho\concat j)}]$
into groups of~$m$ consecutive ones and for each group take the smallest
interval that surrounds its members.

In symbols:
for each $k<m$ the interval 
$[a^{t_{i+1}}_{\varphi_m(\rho\concat (km))},
  b^{t_{i+1}}_{\varphi_m(\rho\concat ((k+1)m-1))}]$ will be a term
of~$A^{t^*}$ and its index will be the value of~$\psi_m$ at $\rho\concat k$.

We also add $\Im\psi_m$ to $G$ and in this way ensure that $\psi_m$
will be $m$-dense for $G$ and $T_m\cup\{t^*\}$.
\end{proof}

We now turn to the task of avoiding having to add points to our linear order
when we do not want to, that is, we want ensure that we can 
achieve property~\eqref{embed.precise} (on page~\pageref{embed.precise})
of the embedding.
It is here that we define the notion of equivalence, promised in that
property.

We introduce some notation:
let $F\subseteq\omega$ and let $A,B\in\Gamma$.

We say that $A$ refines $B$ modulo~$F$, and we write $A\preceq_F B$,
if for every term of~$A$ with $[a,b]\cap F\neq\emptyset$
there is a term $[c,d]$ of~$B$ such that $[a,b]\cap F\subseteq[c,d]$

We say that $A$ and $B$ are equivalent modulo~$F$, written $A\equiv_FB$,
if for every $n\in F$ there are terms $[a,b]$ of~$A$ and $[c,d]$ of~$B$
such that $n\in[a,b]\cap F$ and $[a,b]\cap F=[c,d]\cap F$.

\begin{lemma}\label{lemma.compare}
Let $\trip{F}{T}{\calA_T}$ be a \goodt/, let $t\in T$ and $A\in\Gamma$.
Then there is $F_t\subseteq F$ such that $\trip{F_t}{T}{\calA_T}$
is a \goodt/ and $A\preceq_{F_t} A^t$ or $A^t\preceq_{F_t} A$; in addition
if $t$~has a direct $\prec$-predecessor~$s$ then we can even achieve 
``$A\preceq_{F_t} A^s$ or $A^t\preceq_{F_t} A$''.
\end{lemma}

\begin{proof}
Write $T$ as the union of an increasing sequence~$\omegaseq[m]T$ of finite
sets such that $0,1,t\in T_0$ (and also $s\in T_0$ if present). 
Upon applying Lemmas~\ref{lemma.split} and~\ref{lemma.disjoint.dense}
we may assume that $F$~does not meet consecutive intervals of~$A^t$,
and that we have a sequence~$\omegaseq[m]\varphi$ of functions
such that $\varphi(m)$ is $(m+1)(m+2)$-dense for~$F$ and~$T_m$, and
$\max\Im\varphi_m<\min\Im\varphi_{m+1}$ for all~$m$.
We also assume $Y=\bigcup_{m\in\omega}\Im\varphi_m$.

Enumerate $T_m$ in decreasing order as $\langle t^m_i<l_m\rangle$, and
for every~$m$ let $i_m$ be the index of~$t$.
Abbreviate $t^m_{i_m}$ as~$t_m$ and $t^m_{i_m+1}$ as~$s_m$ 
(so $s_m=s$ for all~$m$ if $s$~is present).

We fix $m$ for a moment and for every $\rho\in\functions{i_m}{((m+1)(m+2))}$ 
we take a term $[a^m_\rho,b^m_\rho]$ of~$A$ such that 
$$
J^m_\rho=\bigl\{j<(m+1)(m+2):
\bA{s_m}{\varphi(\rho\concat j)}\subseteq [a^m_\rho,b^m_\rho]\bigr\}
$$ 
has maximum cardinality.
Divide $\functions{i_m}{((m+1)(m+2))}$ into two parts:
$R_m=\{\rho:\card{J^m_\rho}\ge m\}$ and its complement~$S_m$.

The proof of Lemma~\ref{lemma.split} gives us a subfunction~$\phi_m$
of~$\varphi_m\restr\functions{\le i_m}{((m+1)(m+2))}$ whose domain 
is $(m+1)(m+2)/2$-branching and such 
that $X_m=\dom\phi_m\cap\functions{i_m}{((m+1)(m+2))}$
is a subset of~$R_m$ or of~$S_m$.

In case $X_m\subseteq R_m$ we define a set $F^m_t$ as follows:
$$
F^m_t=F\cap \bigcup\bigl\{ \bA{s_m}{\varphi(\rho\concat j)}:
   \rho\in X_m\text{ and } j\in J^m_\rho\bigr\}
$$
We extend $\phi_m$ to a subfunction $\psi_m$ of~$\varphi_m$ by adding
$$
\bigl\{\rho\in\dom\varphi_m: (\exists\sigma\in X_m)(\exists j\in J^m_\sigma)
 (\sigma\concat j \subseteq\rho)\bigr\}
$$
to its domain and using the values of~$\varphi_m$ at those points.
The resulting function is (more than) $m$-dense for $F^m_t$ and~$T_m$.
Also, if $n\in F^m_t$ then there are $\rho\in X_m$ and $j\in J^m_\rho$
such that 
$n\in \bA{s_m}{\varphi(\rho\concat j)} \subseteq \bA t{\varphi(\rho)}$
and, by definition,
$F^m_t\cap\bA t{\varphi(\rho)}\subseteq[a^m_\rho,b^m_\rho]$.
This shows that if $F^m_t$ were to contribute to $F_t$ it would also
witness $A^t\preceq_{F_t}A$.

Thus, if the situation $X_m\subseteq R_m$ occurs infinitely often then we 
can build an~$F_t$ such that $A^t\preceq_{F_t}A$.

In the other case we get $X_m\subseteq S_m$ infinitely (even cofinitely) often.
We shall build an~$F_t$ that will satisfy $A\preceq_{F_t} A^t$ and even
$A\preceq_{F_t} A^s$ if $s$~is present.

Consider an $m$ such that $X_m\subseteq S_m$ and fix $\rho\in X_m$.
For each term~$[a,b]$ of~$A$ the set 
$\{j:\bA{s_m}{\varphi(\rho\concat j)}\subseteq [a,b]\}$ has at most 
$m-1$~elements; as $[a,b]$ is an interval these are consecutive elements.
This means that $[a,b]$ can intersect at most $m+1$ of these intervals:
at most $m-1$ in the interior and possibly two more that merely overlap at the 
ends. 
We use the intervals indexed by $X_m$ and $I=\{(m+2)(j+1):j<m\}$ to 
define~$F^m_t$: 
$$
F^m_t=F\cap \bigcup\bigl\{ \bA{s_m}{\varphi(\rho\concat i)}:
           \rho\in X_m\text{ and } i\in I\bigr\}
$$
the same formula as in the case `$X_m\subseteq R_m$' with
$J^m_\rho$ replaced by~$I$.
Now if $[a,b]$ is a term of~$A$ and $n\in F^m_t\cap[a,b]$ then there
are one $\rho\in X_m$ and one $i\in I$ such that 
$n\in\bA{s_m}{\varphi(\rho\concat i)}$
and the latter is also the only interval of that form that $[a,b]$
intersects.
It follows automatically that 
$$
F^m_t\cap[a,b]\subseteq\bA{s_m}{\varphi(\rho\concat i)}
               \subseteq\bA t{\varphi_m(\rho)}.
$$
Thus, if we let $F_t$ be the union of these $F^m_t$ then we achieve 
$A\preceq_{F_t}A^t$ and even $A\preceq_{F_t}A^s$ if $s$~is present. 
\end{proof}

\begin{lemma}\label{lemma.insert}
Let $\trip{F}{T}{\calA_T}$ be a \goodt/ and $A\in\Gamma$.
Then there are $G\subseteq F$ and an extension~$T^*$ of~$T$ by at most one 
point~$t^*$ such that $\trip{G}{T^*}{\calA_{T^*}}$ is a \goodt/
and $A\equiv_{G}A_t$ for some $t\in T^*$. 
\end{lemma}

\begin{proof}
We apply Lemma~\ref{lemma.compare} countably many times and 
Lemma~\ref{lemma:diagonal} once so that we can assume that for every~$t\in T$
there is a cofinite subset~$F_t$ of~$F$ such that $A\preceq_{F_t}A^t$ 
or $A^t\preceq_{F_t}A$ and even $A\preceq_{F_t}A^s$ or $A^t\preceq_{F_t}A$
if $t$ has a direct $\prec$-predecessor~$s$.

We divide $T$ into $S_0=\{t:A^t\preceq_{F_t} A\}$ 
and $S_1=\{t:A\preceq_{F_t} A^t\}$.
Note that $0\in S_0$ by default.

We need to consider several cases.

\begin{case}{$S_0$ has a maximum and $S_1$ has a minimum}
Note that by the condition on direct predecessors these must be identical,
say $t=\max S_0=\min S_1$.
Then one verifies that $A\equiv_{F_t}A^t$.
\end{case}

\begin{case}{$S_1$ is empty}
In this case we have $A^1\preceq_{G}A$ and we can thin out $F$ to a set~$G$
such that $A^1\equiv_G A$; then $\trip{G}{T}{\calA_T}$ is a \goodt/. 
\end{case}

For the other cases we write $T$ as the union of an increasing 
sequence~$\omegaseq[m]T$ of finite sets such that $0,1\in T_0$;
as before we take the decreasing enumeration $\langle t^m_i:I<l_m\rangle$
of~$T_m$. For each~$m$ we let $i_m$ be such that $t^m_{i_m}\in S_1$
and $t^m_{i_m+1}\in S_0$; we denote these two points by~$t_m$ and~$s_m$
respectively.   

Furthermore we choose $\omegaseq[m]\varphi$ as in 
Lemma~\ref{lemma.disjoint.dense} so that $\varphi_m$~is
$m$-dense for $F$ and $T_m$ and such 
that $\Im\varphi_m\subseteq F_{t_m}\cap F_{s_m}$.

Fix $m$ for a moment.
We know that $A^{s_m}\preceq_{F_{s_m}} A\preceq_{F_{t_m}}A^{t_m}$;
this implies that for every $\rho\in\functions{i_m}m$ and every $j<m$ there
is a term~$[a,b]$ of~$A$ such that
$$
\bA{s_m}{\varphi(\rho\concat j)}\cap F
  \subseteq[a,b]\cap F  \subseteq\bA{t_m}{\varphi(\rho)}\cap F
\eqno(*)
$$
indeed, $[a,b]$ is found by an application of $A^{s_m}\preceq_{F_{s_m}} A$
and $\bA{t_m}{\varphi(\rho)}$ is the only
possible term of~$A^{t_m}$ that can help witness $A\preceq_{F_{t_m}}A^{t_m}$.

We put 
$G_m=F\cap\bigcup_\rho\bA{s_m}{\varphi(\rho\concat0)}$,
where $\rho$~runs through $\functions{i_m}m$.
We can define two functions~$\phi_m$ and $\psi_m$ 
on~$\functions{\le l_m-1}m$, as follows.
\begin{enumerate}
\item If $|\rho|<i_m$ then $\phi_m(\rho)=\psi_m(\rho)=\varphi_m(\rho)$.
\item If $|\rho|=i_m$ then $\phi_m(\rho)=\varphi_m(\rho)$
                      and $\psi_m(\rho)=\varphi_m(\rho\concat0)$.
\item If $|\rho|>i_m$, say $\rho=\varrho\concat\sigma$, with $|\varrho|=i_m$,
      then $\phi_m(\rho)=\psi_m(\rho)=\varphi(\varrho\concat0\concat\sigma)$.
\end{enumerate}
So, in $\phi_m$ we skip level~$i_m+1$ of the domain of~$\varphi_m$
and in~$\psi_m$ we skip level~$i_m$.
The effect is that $\phi_m$ is $m$-dense for $T_m\setminus\{s_m\}$ and $G_m$,
whereas $\psi_m$~is $m$-dense for $T_m\setminus\{t_m\}$ and~$G_m$.

In addition we have made sure that
$A^{s_m}\equiv_{G_m}A\equiv_{G_m}A^{t_m}$.

We let $G=\bigcup_mG_m$ and consider the remaining cases in turn.

\begin{case}{$S_0$ has no maximum and $S_1$ has a minimum, say $t=\min S_1$}
In this case we know that $t_m=t$ cofinitely often.
If we drop the finitely many $G_m$ for which $t\neq t_m$ then we 
achieve $A\equiv_GA^t$.
Moreover $\trip{G}{T}{\calA_T}$ is a \goodt/, as witnessed by the 
functions~$\phi_m$.   
\end{case}

\begin{case}{$S_0$ has a maximum and $S_1$ has no minimum, say $t=\max S_0$}
In this case we know that $t_m=s$ cofinitely often.
If we drop the finitely many $G_m$ for which $s\neq t_m$ then we 
achieve $A^s\equiv_GA$.
Moreover $\trip{G}{T}{\calA_T}$ is a \goodt/, as witnessed by the 
functions~$\psi_m$.   
\end{case}

\begin{case}{$S_0$ has no maximum and $S_1$ has no minimum}
This case necessitates adding a new point, $t^*$, to $T$ to form~$T^*$
and we insert~$t^*$ into the gap formed by~$S_0$ and~$S_1$.
We then redefine $\phi_m$ on level~$i_m$ so that
its value at~$\rho$ becomes the index of the term of~$A$ that was chosen
to satisfy inclusions~$(*)$.
The new $\phi_m$ is $m$-dense for $\{t^*\}\cup T_m\setminus\{s_m,t_m\}$
and~$G_m$; this establishes that $\trip{G}{T^*}{\calA_{T^*}}$
is a \goodt/.\qedhere
\end{case}
\end{proof}

Repeated application of these lemmas will prove the following theorem,
where we extend the notion of equivalence to (ultra)filters:
if $p$ is an (ultra)filter on~$\omega$ then 
$A\equiv_p B$ means that $A\equiv_FB$ for some $F\in p$.

\begin{theorem}[$\CH$]\label{mainLemma} 
Let $T$ be a linear order of cardinality at most~$\aleph_1$ that has a maximum
and no $\orpr\omega\omega$-gaps.
Then one can find a subcollection 
$\calA_T = \{ A_t : t\in T\}$ of~$\Gamma$ and a P-point ultrafilter~$p$ 
on~$\omega$ such that 
\begin{enumerate}
\item $\trip{p}{\calA_T}{T}$ is a \goodt/
\item for all $A\in \Gamma$, there is a $t\in T$ such that $A\equiv_p A_t$.\qed
\end{enumerate}
\end{theorem}

\section{Finding many different continua}

In this section we shall use Theorem~\ref{mainLemma} 
(and hence the Continuum Hypothesis)
to find $2^\cee$~many different subcontinua of~$\Hstar$.

We shall apply the theorem to the following type of linearly ordered sets
\begin{enumerate}
\item cardinality at most $\aleph_1$
\item no $\orpr\omega\omega$-gaps
\item cofinality $\aleph_0$ (in particular: no maximum)
\end{enumerate}
In keeping with our use of the vernacular we shall call this a \emph{mean}
linear order.

\subsection{One continuum}

Let $T$ be a mean linear order.
We order $T^+=T\cup\{T\}$ ordered by stipulating that $t\prec T$ 
for all~$t\in T$.
We apply Theorem~\ref{mainLemma} to~$T^+$ to obtain a 
family~$\calA_T=\{A^t_p:t\in T^+\}$ and a P-point~$p$ satisfying the conditions
of that theorem.
We define
$$
K_T=\cl\bigcup_{t\in T}A^t_p,
$$
as announced in the beginning of Section~\ref{sec:getting}.

We list some properties of $K_T$ and the individual continua~$A^t_p$.

\begin{lemma}\label{lemma.layer}
For every~$t\neq\min T$ there is a layer~$L^t_p$ of~$A^t_p$ such that 
$\bigcup_{s\prec t}A^s_p\subseteq L^t_p$.  
\end{lemma}

\begin{proof}
Lemma~6.2 of~\cite{Hart} establishes that $A^s_p$~is contained in a layer 
of~$A^t_p$ whenever $s\prec t$; because $\calA_T$ is a chain this layer
is independent of~$s$.  
We need the assumption $t\neq\min T$ to ensure that we actually
have points below~$t$.
\end{proof}

\begin{lemma}
Every $A^t_p$ is nowhere dense in $K_T$ and 
$\bigcup_{t\in T}L^t_p=\bigcup_{t\in T}A^t_p$.  
\end{lemma}

\begin{proof}
Given $t\in T$ there is $s\in T$ such that $t\prec s$.
Then $A^t_p\subseteq L_s$, which establishes the equality of the two unions.

Because $L_s$~is nowhere dense in~$A^s_p$ this also implies
that $A^t_p$~is nowhere dense in~$K_T$.
\end{proof}

\begin{lemma}
$K_T$ is indecomposable.  
\end{lemma}

\begin{proof}
The proof is implicit in~\cite{MR1217484} and~\cite{Hart} as part
of a construction of an indecomposable subcontinuum of~$\Hstar$
called~$K_9$ in the latter paper.

Let $L$ be a proper subcontinuum of~$K_T$.
Note that because each~$L^t_p$ is indecomposable we know that
$L^t_p\subseteq L$ or $L\subseteq L^t_p$ for all~$t$ such 
that $L\cap L^t_p$~is nonempty.
Since it is impossible that $L^t_p\subseteq L$ for all~$t$ (otherwise $L=K_T$)
it follows that $L\cap\bigcup_{t\in T}L^t_p=\emptyset$ or $L\subseteq L^t_p$
for some~$t$.
In either case $L$~is nowhere dense in~$K_T$.
\end{proof}

\begin{lemma}
Every $A^t_p$ is a P-set in~$\Hstar$ as is every~$L^t_p$, for $t\neq\min T$.
\end{lemma}

\begin{proof}
The preimage of $A^t_p$ under the parametrizing map $q:\Mstar\to\Hstar$
consists of~$\II_v$, the point~$1_{v-1}$ and the point~$0_{v+1}$, where
$v$~is such that $A^T_p=[a^t_v,b^t_v]$. 
This makes the preimage a P-set, as $\pi$~is closed this implies that $A^t_p$~is
a P-set as well.  

It suffices to show that $L^t_p$ is not a countable cofinality layer in~$A^t_p$
if $t\neq\min T$.
If $L^t_p$ were such a layer then one of the open intervals
with~$L^t_p$ as its end layer, call it~$I$, would be an $F_\sigma$-set
such that $I\cap L^t_p=\emptyset$ and $L^t_p\subseteq\cl I$.
Now let $s\prec t$; then $A^s_p$ is a P-set and $A^s_p\cap I=\emptyset$.
It follows that $A^s_p\cap\cl I=\emptyset$ as well, which contradicts
 $L^t_p\subseteq\cl I$.
\end{proof}

\subsection{Consequences of homeomorphy}
\label{subsec:consequences}

Let $T$ and $S$ be two mean linear orders.
We assume we have families $\calA_T$ and $\calA_S$ and P-points $p$ and $q$
respectively as in Theorem~\ref{mainLemma}.
We write $F_T=\bigcup_{t\in T}A^t_p$ and $F_S=\bigcup_{s\in S}A^s_q$ and
let $K_T=\cl F_T$ and $K_S=\cl F_S$.
We retain the notations $L^t_p$ and~$L^t_q$ respectively for the layers
from Lemma~\ref{lemma.layer}.
We assume that $K_T$ and $K_S$ are homeomorphic and let $f:K_T\to K_S$
be a homeomorphism.

\begin{lemma}\label{lemma.FT=FS}
$f[F_T]=F_S$.  
\end{lemma}

\begin{proof}
Let $t\in T$.
Because the P-set $f[A^t_p]$ is in the closure of the $F_\sigma$-set~$F_S$ 
it must actually intersect that set.
Thus there is an~$s\in S$ such that $f[A^t_p]\cap A^s_q\neq\emptyset$
and hence $f[A^t_p]\cap L^r_q\neq\emptyset$ whenever $s\prec r$ in~$S$.
It follows that $f[A^t_p]\subseteq L^r_q$ or $L^r_q\subseteq f[A^t_p]$
for all~$r\succ s$ and because $f[A^t_p]$ is nowhere dense in~$K_S$
we must have $f[A^t_p]\subseteq L^r_q$ for a final segment of~$r$ in~$S$.

This shows that $f[F_T]\subseteq F_S$ and, using $f^{-1}$ instead of~$f$,
we can also deduce that $F_S\subseteq f[F_T]$.
Thus we find that $F_T$ is mapped onto~$F_S$ by~$f$.
\end{proof}

Our aim is now to show that $T$ and $S$ have isomorphic final segments. 

Let $T'=\{t\in T:(\exists s\in S)(A^s_q\subseteq f[L^t_p])\}$
and, symmetrically, let  
$S'=\{s\in S:(\exists t\in T)(f[A^t_p]\subseteq L^s_q)\}$.
We shall show that $T'$ and $S'$ are isomorphic by showing that $f$~induces
an isomorphism between the families $\{L^t_p:t\in T'\}$ 
and $\{L^s_q:s\in S'\}$ (ordered by inclusion).

\smallskip
Let $t\in T'$ and consider $f[A^t_p]$; this is a decomposable continuum
and hence it is an interval of some standard subcontinuum.
We shall find $A\in\Gamma$ such that $f[A^t_p]$ is in fact an interval of~$A_q$.
To this end let $\bigl<[c_n,d_n]:n\in\omega\bigr>$ be a sequence of closed
intervals with $d_n=c_{n+1}$ for all~$n$ and let $r\in\Nstar$ be such
that $f[A^t_p]$ is an interval of~$[c_r,d_r]$.
For every $n$ let $i_n=\lfloor c_n\rfloor$ and $j_n=\lceil d_n\rceil$.

There is a member $R$ of~$r$ such that if $n<m$ in~$R$ then $j_n<i_m$ and
in this case we can assume that 
$\bigl<[i_n,j_n]:n\in R\bigr>$ is a subsequence of some $A\in\Gamma$.
It is clear that $[c_r,d_r]\subseteq[i_r,j_r]$ and it is also true that 
$q\in f[A^t_p]\subseteq[c_r,d_r]$; together these statements imply
that $A_q=[i_r,j_r]$, so that $f[A^t_p]$ is indeed an interval of~$A_q$. 

Now let $s_t\in S$ be such that $A\equiv_q A^{s_t}$ and fix some $s\in S$ 
such that $A^s_q\subseteq f[L^t_p]$.
We claim that $s\prec s_t$.
Indeed, if $s_t\preceq s$ then we find that 
$A^{s_t}_q\subseteq A^s_q\subseteq f[L^t_p]$ and hence that $A^{s_t}_q$ is nowhere
dense in $f[A^t_p]$ and hence in $A_q$, which contradicts $A\equiv_q A^{s_t}$.
Thus we find that $A^s_q\subseteq L^{s_t}_q$ and hence 
that $f[L^t_p]\cap L^{s_t}_q\neq\emptyset$.
But $f[L^t_p]$ is a layer of~$f[A^t_p]$ and hence of~$A_q\cup A^{s_t}_q$, as 
is~$L^{s_t}_q$ of course.
But then we must have $f[L^t_p]=L^{s_t}_q$.

Since $L^{t_1}_p$ is nowhere dense in $L^{t_2}_p$, whenever 
$t_1\prec t_2$ in~$T$,
the map $t\mapsto s_t$ from~$T'$ to~$S'$ is strictly increasing; 
that it is surjective follows by interchanging $S'$ and~$T'$ and 
considering~$f^{-1}$.

This shows that $T'$ and $S'$ are isomorphic.

\subsection{Many ordered sets}

We define a family of $2^{\aleph_1}$ many linear orders of countable cofinality
and without isomorphic final segments.

For a set $X$ of countable limit ordinals we define a linear order~$L_X$
by inserting upside-down copies of~$\omega$ into~$\omega_1$, one between
$\alpha$ and $\alpha+1$ for every $\alpha\in X$.
More formally we let 
$$
L_X=\{\orpr\alpha m\in\omega_1\times\omega: \alpha\notin X \rightarrow m=0\}
$$
ordered by $\orpr\alpha m\prec\orpr\beta n$ if
1)~$\alpha\in\beta$, or
2)~$\alpha=\beta$ and $m=0<n$, or
3)~$\alpha=\beta$ and $m>n>0$.

\begin{proposition}
$L_X$ and $L_Y$ are isomorphic iff $X=Y$.  
\end{proposition}

\begin{proof}
Let $f:L_X\to L_Y$ be an isomorphism.
We show by induction that $f(\orpr\alpha0)=\orpr\alpha0$ for every limit
ordinal~$\alpha$ as well as $\alpha\in X$ iff $\alpha\in Y$.

In both $L_X$ and $L_Y$ the point $\orpr\omega0$ has $\omega\times\{0\}$
as its set of predecessors and so $f(\orpr\omega0)=\orpr\omega0$.
Assume $\alpha$ is a limit and that $f(\orpr\beta0)=\orpr\beta0$
for all limits below~$\alpha$.
If $\alpha$ is a limit of limits then in both ordered sets we have
$\orpr\alpha0=\sup\{\orpr\beta0:\beta\in\alpha, \beta$~is a limit$\}$
and hence $f(\orpr\alpha0)=\orpr\alpha0$.

Next assume $\alpha=\beta+\omega$ for a limit~$\beta$.
If $\beta\notin X$ then $\orpr{\beta+1}0$ is the direct successor in~$L_X$ 
of~$\orpr\beta0$, hence $\orpr\beta0$ must have a direct successor in~$L_Y$
as well.
From this it follows that $\beta\notin Y$ and $f(\orpr{\beta+n}0)=\orpr{\beta+1}0$
for all~$n\in\omega$ and hence also $f(\orpr\alpha0)=\orpr\alpha0$.

If $\beta\in X$ then the interval $\bigl(\orpr\beta0,\orpr\alpha0\bigr)$
has the same order type as~$\Z$, the set of integers.
Now the interval $\bigl(\orpr\beta0, \orpr\beta1\bigr]$ is infinite and every
point in it has a direct predecessor. 
This means that $f(\orpr\beta1)\prec\orpr\alpha0$ and hence that $\orpr\beta0$
does not have a direct successor in~$L_Y$ and hence that $\beta\in Y$.
It follows that $f$~maps the interval $\bigl(\orpr\beta0,\orpr\alpha0\bigr)$
isomorphically onto the corresponding interval of~$L_Y$ and that
$f(\orpr\alpha0)=\orpr\alpha0$. 
\end{proof}

From $L_X$ we define $T_X$ to be the ordered sum of $\omega$ copies of~$L_X$:
$$
T_X=\omega\times L_X
$$
ordered lexicographically.
Now note that the points $\bigorpr n{\orpr00}$ are the only ones in~$T_X$
whose sets of predecessors have cofinality~$\aleph_1$.

Thus, if $f$ is an isomorphism between final segments of some~$T_X$
and~$T_Y$ then there an isomorphism~$g$ between final segments of~$\omega$
such that $f(n,0,0)=(g(n),0,0)$ for all in the final segment on the $T_X$-side.
For each such~$n$ the map~$f$ then maps $\{n\}\times L_X$
isomorphically onto $\{g(n)\}\times L_Y$.
It follows that $X=Y$.

This then provides us with our family of $2^{\aleph_1}$ many linear orders,
indexed by the family of sets of countable limit ordinals.

This proves the following theorem and with it the existence
of a family of $2^\cee$ many mutually nonhomeomorphic subcontinua of~$\Hstar$.

\begin{theorem}[$\CH$]\label{thm.many.orders}
There is a family of $2^\cee$ mean linear orders such that no two members
have isomorphic final segments.\qed
\end{theorem}

\subsection{Summary: two families of continua}

The combination of subsection~\ref{subsec:consequences} and 
Theorem~\ref{thm.many.orders} tells us that
$\{K_{T_X}:X$~a set of countable limit ordinals$\}$ is a family of $2^\cee$~many
indecomposable subcontinua of~$\Hstar$ that are mutually non-homeomorphic.

To get a family of $2^\cee$~many decomposable continua use 
Lemma~\ref{lemma.K-int-L} to deduce that in our construction
the continuum~$K_T$ is actually a layer of the `top continuum'~$A_{T^+}$.
Indeed, $K_T$~is a subset of some layer~$L$ of~$A_{T^+}$; if it were a proper
subset then there would be a standard subcontinuum~$M$ with
$K_T\subseteq M\subseteq L$.
As in subsection~\ref{subsec:consequences} we could then find~$A\in\Gamma$
such that $M$~is an interval of~$A$; yet there would be no $t\in T^+$
such that $A\equiv_p A_t$.

Our second family is now obtained by taking for every set $X$ of countable
limit ordinals the interval $[a_X,K_{T_X}]$ of the standard 
subcontinuum~$A_{T_X^+}$, where $a_X$~is the initial point of~$A_{T_X^+}$
as described in subsection~\ref{subsec:subcont}.
These decomposable continua are mutually non-homeomorphic because a 
homeomorphism between $[a_X,K_{T_X}]$ and~$[a_Y,K_{T_Y}]$ will have to map $a_X$
to~$a_Y$ (as these are the unique end points) and $K_{T_X}$ onto~$K_{T_Y}$,
the latter is not possible if $X\neq Y$.

\begin{remark}
The family in~\cite{MR2823685} consists of standard subcontinua.
By one of the results in~\cite{MR1282963} $\CH$~implies that all standard
subcontinua are homeomorphic.
Thus there is a striking difference between the effects of $\CH$ and $\neg\CH$
on the structure of family of standard subcontinua.

Our result shows that under~$\CH$ each standard subcontinuum has a rich variety
of layers and intervals.  
We leave as an open question how rich this variety is in $\ZFC$ alone.
\end{remark}

\section{A first-countable continuum}
\label{sec.bell}

\subsection{Bell's graph}

A major ingredient in our construction is Bell's graph, constructed
in~\cite{MR677860}.
It is a graph on the ordinal~$\omega_2$, represented by a symmetric
subset~$E$ of~$(\omega_2)^2$.
The crucial property of this graph is that there is \emph{no} map 
$\varphi:\omega_2\to\Pow(\N)$ that represents this graph,
where \emph{$\varphi$ represents~$E$} if
$\orpr\alpha\beta\in E$ if and only if $\varphi(\alpha)\cap\varphi(\beta)$ is
infinite.

Bell's graph exists in any forcing extension in which $\aleph_2$~Cohen 
reals are added; for the reader's convenience we shall,
in subsection~\ref{subsec:the-graph} below, describe the
construction of~$E$ and adapt Bell's proof so that it applies to continuous 
maps defined on~$\Hstar$.
The proof shows that a similar graph also exists in the extension by
$\aleph_2$~random reals.

\subsection{Building $\CE$}

Our starting point is a connected version of the Alexandroff double of the
unit interval, devised by Saalfrank~\cite{MR0031710}.
We topologize the unit square as follows.
\begin{enumerate}
\item a local base at points of the form $\orpr x0$ consists of the sets
      $$
       U(x,0,n)=(x-2^{-n},x+2^{-n})\times[0,1] \setminus
                \{x\}\times[2^{-n},1]
      $$
\item a local base at points of the form $\orpr xy$, with $y>0$ consists
      of the sets
      $$
       U(x,y,n) = \{x\}\times(y-2^{-n},y+2^{-n})
      $$
\end{enumerate}
We call the resulting space the \textsl{connected comb} and denote it by~$C$.
It is straightforward to verify that $C$ is compact, Hausdorff and connected;
it is first-countable by definition.

For each $x\in[0,1]$ and positive $a$ we define the following
cross-shaped closed subset of~$C^2$:
$$
D_{x,a}= 
\bigl(\{x\}\times[a,1]\times C\bigr)\cup\bigl(C\times\{x\}\times[a,1]\bigr)
$$
We note the following two properties of the sets $D_{x,a}$
\begin{enumerate}
\item if $a<b$ then $D_{x,b}$ is in the interior of~$D_{x,a}$, and
\item if $x\neq y$ then $D_{x,a}\cap D_{y,a}$ is the union of two squares:
       $\{x\}\times[a,1]\times\{y\}\times[a,1]$ and 
       $\{y\}\times[a,1]\times\{x\}\times[a,1]$.
\end{enumerate}

Next take any $\aleph_2$-sized subset of~$[0,1]$ and index it (faithfully)
as $\{x_\alpha:\alpha<\omega_2\}$.
We use this indexing to identify $E$ with the subset 
$\{\orpr{x_\alpha}{x_\beta}:\orpr\alpha\beta\in E\}$ of the unit square.
We remove from $C^2$ the following open set:
$$
\bigcup_{\orpr xy\notin E}
\Bigl(\bigl(\{x\}\times(0,1]\times\{y\}\times(0,1]\bigr) \cup
\bigl(\{y\}\times(0,1]\times\{x\}\times(0,1]\bigr)\Bigr)
$$
The resulting compact space we denote by $\CE$.
Observe that the intersections $D_{x_\alpha,a}\cap \CE$ represent~$E$ in the 
sense that $D_{x_\alpha,a}\cap D_{x_\beta,a}\cap \CE$ is nonempty if and only
if~$\orpr\alpha\beta\in E$.
We write $D^E_{x,a}=D_{x,a}\cap \CE$.

\subsection{$\CE$ is (arcwise) connected}  

To begin: the square~$S$ of the base line of~$C$ is a subset of~$\CE$ and 
homeomorphic to the unit square so that it is (arcwise) connected.

Let $\langle x,a,y,b\rangle$ be a point of~$\CE$ not in~$S$.
If, say, $a=0$ then 
$\bigl\{\orpr x0\bigr\}\times\bigl(\{y\}\times[0,b]\bigr)$ is an arc 
in~$\CE$ that connects $\langle x,0,y,b\rangle$ to the 
point~$\langle x,0,y,0\rangle$ in~$S$.
If $a,b>0$ then $\orpr xy\in E$, so the whole square
$\{x\}\times[0,1]\times\{y\}\times[0,1]$ is in~$\CE$ and it provides
us with an arc in~$\CE$ from $\langle x,a,y,b\rangle$ 
to $\langle x,0,y,0\rangle$.

\smallskip

We find that $\CE$ is a first-countable continuum.

\subsection{$C_E$ is not an $\Hstar$-image}

Assume $h:\Hstar\to \CE$ is a continuous surjection and consider, 
for each~$\alpha$, the sets $D^E_{x_\alpha,\frac34}$ 
and $D^E_{x_\alpha,\frac12}$.

Using standard properties of $\betaH$, see~\cite{Hart}*{Proposition~3.2},
we find for each~$\alpha$ a sequence 
$\bigl<(a_{\alpha,n},b_{\alpha,n}):n\in\N\bigr>$
of open intervals with rational endpoints, 
and with $b_{\alpha,n}<a_{\alpha,n+1}$ for all~$n$, 
such that 
$h\preim[D^E_{x_\alpha,\frac34}]\subseteq 
    \Ex O_\alpha\cap\Hstar \subseteq h\preim[D^E_{x_\alpha,\frac12}]$,
where $O_\alpha=\bigcup_n(a_{\alpha,n},b_{\alpha,n})$
and $\Ex O_\alpha=\betaH\setminus\cl(\HH\setminus O_\alpha)$.

Because the intersections of the sets~$D^E_{x_\alpha,a}$ represent~$E$
the intersections of the $O_\alpha$ will do this as well:
the conditions `$O_\alpha\cap O_\beta$ is unbounded' 
and `$\orpr\alpha\beta\in E$' are equivalent.

In the next subsection we show that for (many) $\orpr\alpha\beta$
this equivalence does not hold and that therefore $\CE$ is not a continuous
image of~$\Hstar$.

Note also that our continuum is not an $\Nstar$-image either: 
if $g:\Nstar\to \CE$ were continuous and onto we could use clopen subsets
of~$\Nstar$ and their representing infinite subsets of~$\N$ to contradict
the unrepresentability property of~$E$.

\subsection{Building the graph}
\label{subsec:the-graph}

We follow the argument from~\cite{MR677860} and we rely on Kunen's book 
\cite{MR597342}*{Chapter~VII} for basic facts on forcing.
We let $L=\{\orpr\alpha\beta\in(\omega_2)^2:\alpha\le\beta\}$ 
and we force with the partial order $\Fn(L,2)$ of finite partial functions 
with domain in~$L$ and range in~$\{0,1\}$.
If $G$~is a generic filter on~$\Fn(L,2)$ then we let 
$E=\{\orpr\alpha\beta: \bigcup G(\alpha,\beta)=1$ or 
$\bigcup G(\beta,\alpha)=1\}$.

To show that $E$ is as required we take a nice name~$\fname F$ for a function 
from~$\omega_2$ to~$(\Q^2)^\omega$ that represents a choice of open 
sets~$\alpha\mapsto O_\alpha$ as in above in that  
$F(\alpha)=\bigl<\orpr{a_{\alpha,n}}{b_{\alpha,n}}:n\in\omega\bigr>$
for all~$\alpha$.
As a nice name $\fname F$ is a subset of 
$\omega_2\times\omega\times\Q^2\times\Fn(L,2)$, where for each point
$\langle\alpha,n,a,b\rangle$ the set 
$\{p:\langle\alpha,n,a,b,p\rangle\in\fname F\}$ is a maximal antichain
in the set of conditions that forces the $n$th term of $\fname F(\alpha)$
to be $\orpr ab$.

For each $\alpha$ we let $I_\alpha$ be the set of ordinals that occur in the
domains of the conditions that appear as a fifth coordinate in the elements 
of~$\fname F$ with first coordinate~$\alpha$.
The sets $I_\alpha$ are countable, by the ccc of $\Fn(L,2)$.
We may therefore apply the Free-Set Lemma,
see \cite{MR795592}*{Corollary~44.2},
and find a subset~$A$ of~$\omega_2$ of cardinality~$\aleph_2$ such that
$\alpha\notin I_\beta$ and $\beta\notin I_\alpha$ whenever 
$\alpha,\beta \in A$ and $\alpha\neq\beta$.

Let $p\in\Fn(L,2)$ be arbitrary and take $\alpha$ and~$\beta$ in~$A$ with 
$\alpha<\beta$ and such that $\alpha>\eta$ whenever $\eta$ occurs in~$p$.
Consider the condition $q=p\cup\bigl\{\langle\alpha,\beta,1\rangle\bigr\}$. 
If $q$~forces $O_\alpha\cap O_\beta$ to be bounded in~$[0,\infty)$ then we 
are done: $q$~forces that the equivalence fails at $\orpr\alpha\beta$.

If $q$~does not force the intersection to be bounded we can extend~$q$
to a condition~$r$ that forces $O_\alpha\cap O_\beta$ to be unbounded.
We define an automorphism $h$ of~$\Fn(L,2)$ by changing the value of the
conditions only at~$\orpr\alpha\beta$: from~$0$ to~$1$ and vice versa.
The condition~$p$ as well as the values~$\fname F(\alpha)$ and $\fname F(\beta)$
are invariant under~$h$.
It follows that $h(r)$ extends~$p$ and
$$
h(r)\forces \hbox{$\bigcup\fname G(\alpha,\beta)=0$ and 
   $O_\alpha\cap O_\beta$ is unbounded}
$$
so again the equivalence is forced to fail at~$\orpr\alpha\beta$.

\begin{remark}
The argument above goes through almost verbatim to show that Bell's graph
can also be obtained adding $\aleph_2$ random reals.
When forcing with the random real algebra one needs only consider conditions
that belong to the $\sigma$-algebra generated by the clopen sets of the
product $\{0,1\}^L$; these all have countable supports so that, 
again by the ccc, one can define the sets $I_\alpha$ as before.
The rest of the argument remains virtually unchanged.
\end{remark}

\begin{remark}
Bell's original example from \cite{MR677860} was not easily made connected.
One obtains an essentially equivalent example by taking the square
of the Alexandroff double of the unit interval 
(the subspace $\bigl\{\orpr xi:x\in[0,1], i\in\{0,1\}\bigr\}$ of~$C$) 
and removing the points $\bigl<\orpr x1, \orpr y1\bigr>$ 
with $\orpr xy\notin E$.
\end{remark}


\begin{bibdiv}
  
\begin{biblist}

\bib{MR0251695}{article}{
   author={Arhangel{\cprime}ski{\u\i}, A. V.},
   title={The power of bicompacta with first axiom of countability},
   journal={Soviet Mathematics Doklady},
   volume={10},
   year={1969},
   pages={951\ndash 955},
   note={Russian original: Doklady Akademi\u{\i} Nauk SSSR
         \textbf{187} (1969) 967\ndash 970},
   review={\MR{0251695 (40 \#4922)}},
}


\bib{MR677860}{article}{
   author={Bell, Murray G.},
   title={The space of complete subgraphs of a graph},
   journal={Commentationes Mathematicae Universitatis Carolinae},
   volume={23},
   date={1982},
   number={3},
   pages={525--536},
   issn={0010-2628},
   review={\MR{677860 (84a:54050)}},
}

\bib{MR1058795}{article}{
   author={Bell, Murray G.},
   title={A first countable compact space that is not an $\Nstar$ image},
   journal={Topology and its Applications},
   volume={35},
   date={1990},
   number={2-3},
   pages={153--156},
   issn={0166-8641},
   review={\MR{1058795 (91m:54028)}},
}


\bib{MR2823685}{article}{
   author={Dow, Alan},
   title={Some set-theory, Stone-\v Cech, and $F$-spaces},
   journal={Topology and Applications},
   volume={158},
   date={2011},
   number={14},
   pages={1749--1755},
   issn={0166-8641},
   review={\MR{2823685}},
   doi={10.1016/j.topol.2011.06.007},
}


\bib{MR1282963}{article}{
   author={Dow, Alan},
   author={Hart, Klaas Pieter},
   title={\v Cech-Stone remainders of spaces that look like $[0,\infty)$},
   note={Selected papers from the 21st Winter School on Abstract Analysis
   (Pod\v ebrady, 1993)},
   journal={Acta Universitatis Carolinae. Mathematica et Physica},
   volume={34},
   date={1993},
   number={2},
   pages={31--39},
   issn={0001-7140},
   review={\MR{1282963 (95b:54031)}},
}

\bib{MR1679586}{article}{ 
 author={Dow, Alan}, 
 author={Hart, Klaas Pieter}, 
 title={$\omega^*$ has (almost) no continuous images}, 
 journal={Israel J. Math.}, 
 volume={109}, 
 date={1999}, 
 pages={29--39}, 
 issn={0021-2172}, 
 review={\MR{1679586 (2000d:54031)}}, 
 doi={10.1007/BF02775024}, 
} 

\bib{MR1707489}{article}{
   author={Dow, Alan},
   author={Hart, Klaas Pieter},
   title={A universal continuum of weight $\aleph$},
   journal={Transactions of the American Mathematical Society},
   volume={353},
   date={2001},
   number={5},
   pages={1819--1838},
   issn={0002-9947},
   review={\MR{1707489 (2001g:54037)}},
}
\bib{MR2425747}{article}{
   author={Dow, Alan},
   author={Hart, Klaas Pieter},
   title={A separable non-remainder of $\mathbb{H}$},
   journal={Proceedings of the American Mathematical Society},
   volume={136},
   date={2008},
   number={11},
   pages={4057--4063},
   issn={0002-9939},
   review={\MR{2425747 (2009h:03066)}},
   doi={10.1090/S0002-9939-08-09357-X},
}


\bib{MR795592}{book}{
   author={Erd{\H{o}}s, Paul},
   author={Hajnal, Andr{\'a}s},
   author={M{\'a}t{\'e}, Attila},
   author={Rado, Richard},
   title={Combinatorial set theory: partition relations for cardinals},
   series={Studies in Logic and the Foundations of Mathematics},
   volume={106},
   publisher={North-Holland Publishing Co.},
   place={Amsterdam},
   date={1984},
   pages={347},
   isbn={0-444-86157-2},
   review={\MR{795592 (87g:04002)}},
}



\bib{Hart}{incollection}{
    author={Hart, Klaas~Pieter},
     title={The \v{C}ech-Stone compactification of the Real Line},
      date={1992},
 booktitle={Recent progress in general topology},
    editor={Hu\v{s}ek, Miroslav},
    editor={van Mill, Jan},
 publisher={North-Holland},
   address={Amsterdam},
     pages={317\ndash 352},
    review={\MR{95g:54004}},
}

\bib{MR597342}{book}{
   author={Kunen, Kenneth},
   title={Set theory},
   series={Studies in Logic and the Foundations of Mathematics},
   volume={102},
   note={An introduction to independence proofs},
   publisher={North-Holland Publishing Co.},
   place={Amsterdam},
   date={1980},
   pages={xvi+313},
   isbn={0-444-85401-0},
   review={\MR{597342 (82f:03001)}},
}


\bib{Parovicenko63}{article}{
      author={Parovi{\v{c}}enko, I.~I.},
       title={A universal bicompact of weight $\aleph$},
        date={1963},
     journal={Soviet Mathematics Doklady},
      volume={4},
       pages={592\ndash 595},
        note={Russian original: { \emph{Ob odnom universal{\cprime}nom
  bikompakte vesa~$\aleph$}, Doklady Akademi\u{\i} Nauk SSSR \textbf{150}
  (1963) 36--39}},
      review={\MR{27\#719}},
}

\bib{MR0031710}{article}{
   author={Saalfrank, C. W.},
   title={Retraction properties for normal Haussdorff spaces},
   journal={Fundamenta Mathematicae},
   volume={36},
   date={1949},
   pages={93--108},
   issn={0016-2736},
   review={\MR{0031710 (11,194b)}},
}



\bib{MR1217484}{article}{ 
   author={Zhu, Jian-Ping},                
   title={Continua in ${\bf R}^\ast$}, 
   journal={Topology and Applications},      
   volume={50},    
   date={1993},
   number={2},
   pages={183--197},
   issn={0166-8641},
   review={\MR{1217484 (94c:54050)}},
   doi={10.1016/0166-8641(93)90020-E},
}                             
\end{biblist}
\end{bibdiv}
\end{document}